\documentclass[a4paper,10pt]{amsart}

\usepackage{amssymb}
\usepackage{fullpage}
\usepackage{enumerate}
\usepackage[OT4]{fontenc}
\usepackage{tikz} 
\usepackage{float}
\usetikzlibrary{arrows}

\theoremstyle{definition}
\newtheorem{definition}{Definition}
\theoremstyle{plain}
\newtheorem{theorem}[definition]{Theorem}

\newtheorem{lemma}[definition]{Lemma}

\newtheorem{corollary}[definition]{Corollary}

\theoremstyle{remark}
\newtheorem{remark}[definition]{Remark}
\newtheorem{example}[definition]{Example}

\DeclareMathOperator{\Ass}{Ass}

\DeclareMathOperator{\reg}{reg}
\DeclareMathOperator{\HF}{HF}
\DeclareMathOperator{\HP}{HP}
\DeclareMathOperator{\aHP}{aHP}
\DeclareMathOperator{\aHF}{aHF}
\DeclareMathOperator{\gin}{gin}
\DeclareMathOperator{\iin}{in}
\DeclareMathOperator{\vol}{vol}
\DeclareMathOperator{\conv}{conv}
\DeclareMathOperator{\degrevlex}{degrevlex}

\def\field{\mathbb{K}}
\def\PP{\mathbb{P}}
\def\RR{\mathbb{R}}
\def\TT{\mathbb{T}}

\definecolor{qqqqff}{rgb}{0,0,1}
\definecolor{uuuuuu}{rgb}{0.27,0.27,0.27}
\definecolor{zzttqq}{rgb}{0.6,0.2,0}
\definecolor{xdxdff}{rgb}{0.49019607843137253,0.49019607843137253,1.0}

\let\to\longrightarrow
\let\mapsto\longmapsto

\begin{document}

\title{Asymptotic Hilbert Polynomial and limiting shapes}

\author{Marcin Dumnicki, Justyna Szpond, Halszka Tutaj-Gasi\'nska}

\thanks{Corresponding author: Marcin Dumnicki \\ Keywords: symbolic powers, asymptotic invariants.}

\subjclass{13P10; 14N20}

\begin{abstract}
The main aim of this paper is to provide a method which allows finding
limiting shapes of symbolic generic initial systems of higher-dimensional
subvarieties of $\PP^n$. M. Musta\c t\u a and S. Mayes established a connection
between volumes of complements of limiting shapes and the asymptotic
multiplicity for ideals of points. In the paper we prove a generalization
of this fact to higher-dimensional sets.
\end{abstract}

\maketitle

\section{Introduction}
In what follows, let $\field$ be an algebraically closed field of characteristic zero,
by $\field[\PP^n]=\field[x_1,\dots,x_{n+1}]$ we denote the homeogeneous
coordinate ring of the projective space $\PP^n$.
Let $I$ be a homogeneous radical ideal in $\field[\PP^n]$, let $I^{(m)}$ denote its $m$-th symbolic power, defined
as:
$$I^{(m)} = \field[\PP^n] \cap \bigcap_{Q \in \Ass(I)} (I^m)_{Q},$$
where localizations are embedded in a field of fractions of $\field[\PP^n]$ (\cite{Eis95}).
By the Zariski-Nagata theorem, for a radical homogeneous ideal $I$ in a polynomial ring
over an algebraically closed field, the $m$-th symbolic power $I^{(m)}$ is equal to
$$I^{(m)} = \bigcap_{p \in V(I)} \mathfrak{m}_{p}^m,$$
where $\mathfrak{m}_p$ denotes the maximal ideal of a point $p$, and $V(I)$ denotes the set of zeroes of $I$.
In characteristic zero, the symbolic power can also be described as the set of polynomials, which vanish
to order $m$ along $V(I)$, which (compare \cite{Sul08}) can be written as:
$$I^{(m)} = \left(f : \frac{\partial^{|\alpha|} f}{\partial x^{\alpha}} \in I \text{ for } |\alpha| \leq m-1 \right).$$
Thus symbolic powers are much more geometric in nature than regular powers $I^m$, but algebraically hard to find ---
finding generators of $I^{(m)}$ knowing generators of $I$ is, in many cases, beyond our knowlegde.
In recent years, symbolic powers receive much attention, e.g. \cite{BH,HH,ELS} and references therein. Since
computation of initial degree (the lowest degree of a non-zero form in $I^{(m)}$) or Castelnuovo-Mumford regularity
of $I^{(m)}$ for a given $m$
can be very hard, one idea is to compute asymptotic versions of these invariants, i.e. to measure,
how the initial degree of $I^{(m)}$ or $\reg(I^{(m)})$ grows when $m$ increases.

The sequence $\{I^{(m)}\}_{m}$ can be naturally regarded as a graded sequence of ideals
(by definition, the graded sequence $\{I_m\}_m$ of ideals satisfy $I_m \cdot I_k \subset I_{m+k}$;
in our case $I^{(m)} \cdot I^{(k)} \subset I^{(m+k)}$). To such graded sequences
one can attach asymptotic invariants, see e.g. \cite[Section 2.4.B]{PAG}.
An example of such an invariant is the \emph{limiting shape}, defined by S. Mayes \cite{Mayci}.
This invariant carries information about the geometry of $V(I)$, computing e.g.
its asymptotic initial degree (called also the Waldschmidt constant) or asymptotic regularity (for more details see \cite{Mayci,Maypoints}).

To define the limiting shape, consider first \emph{generic initial ideal} $\gin(I)$ of $I$,
as the initial ideal, with respect to the $\degrevlex$, the degree reverse lexicographical order,
of a generic coordinate change of $I$ (more information about initial ideals and orderings
can be found in any book on Gr\"obner Bases Theory). Galligo \cite{Gal74} assures that for a homogeneous ideal $I$ and a generic
choice of coordinates, the initial ideal of $I$ is fixed, hence the definition of $\gin(I)$ is correct.
Some of the properties of $\gin(I)$, even better than these of the usual initial ideal $\iin(I)$, are listed in \cite{Gre98}.

In the next step consider the sequence of monomial ideals $\gin(I^{(m)})$. The $m$-th symbolic power of a radical
ideal $I$ is saturated (that is, $I^{(m)}:M=I^{(m)}$, for $M=(x_1,\dots,x_{n+1})$; use Nagata-Zariski theorem), hence by Green \cite[Theorem 2.21]{Gre98}
no minimal generator of $\gin(I^{(m)})$ contains the last variable $x_{n+1}$. Therefore
these monomial ideals can be naturally regarded as ideals in $\field[x_1,\dots,x_n]$.
The Newton polytope of a monomial ideal is defined as a convex hull of the set of exponents:
$$P(J) := \conv(\{ \alpha \in \RR^n : x^{\alpha} \in J\}).$$

The limiting shape of an ideal $I$ as above is defined as
$$\Delta(I) = \bigcup_{m=1}^{\infty} \frac{P(\gin(I^{(m)}))}{m},$$
(see Mayes \cite{Maypoints}). So far, these limiting shapes have been found for complete
intersections (Mayes \cite{Mayci}),
points in $\PP^2$ (Mayes \cite{Maypoints} assuming Segre-Hirschowitz-Gimigliano-Harbourne
conjecture), star configurations in $\PP^n$ (the authors with T. Szemberg, \cite{DSST}). The
crucial result, which allows all the above computations, has been observed by Musta\c t\u a
\cite[Theorem 1.7 and Lemma 2.13]{Mus} and Mayes \cite[Proposition 2.14]{Maypoints}. Let
$\Gamma(I)$ be the closure of the complement of $\Delta(I)$ in $\RR^n_{\geq 0}$. If $I$ is a
zero-dimensional radical ideal of $r$ points, then $\Gamma(I)$ is bounded, and
$$\vol(\Gamma(I)) = \frac{r}{n!}.$$

In this paper we propose a generalization of the above theorem for ideals of any dimension. Of course,
the sets $\Gamma(I)$ are then unbounded (with infinite volume), hence the generalization cannot be straightforward.
Our idea is to introduce the asymptotic invariant, called the \emph{asymptotic Hilbert polynomial of $I$}, $\aHP_I$, which
will measure the size of $\Gamma(I)$. As in the known zero-dimensional cases, this invariant allows to already
find the $\Gamma(I)$ (or $\Delta(I)$) in many higher-dimensional cases, and we present some examples.
This invariant is also
a generalization of polynomials $\Lambda_{n,r,s}$ from \cite[Section 2]{DHST}. It also carries some natural properties, similar to that of $\HP_I$, such
as $\aHP_{I \cap J} = \aHP_I+\aHP_J$ for ideals satisfying $V(I)\cap V(J) = \varnothing$.

To define asymptotic Hilbert polynomial, recall that
the Hilbert function $\HF_{I}$ of a homogeneous ideal $I$ is defined as
$$\HF_{I}(t) = \dim_{\field}(\field[\PP^n]_t/I_t).$$
For $t$ big enough the above function behaves as a polynomial, the Hilbert polynomial
$\HP_{I}$ of $I$. We define two new objects.
\begin{definition}
The \emph{asymptotic Hilbert function of $I$}
$$\aHF_I(t):=\lim_{m\to \infty}\frac{\HF_{I^{(m)}}(mt)}{m^n}.$$
\end{definition}
and
\begin{definition}
The \emph{asymptotic Hilbert polynomial of $I$}
$$\aHP_I(t):=\lim_{m\to \infty}\frac{\HP_{I^{(m)}}(mt)}{m^n}.$$
\end{definition}
In the paper, we prove the existence of the first limit for radical ideals (Theorem \ref{ahf is a
limit}), the existence of the second for ideals with linearly bounded symbolic regularity
 (see Definition \ref{LBSR}) in Theorem \ref{ahp
istnieje}. Also we prove that (still for ideals with bounded regularity) $\aHP_I(t) = \aHF_I(t)$
for $t$ big enough (again Theorem \ref{ahp istnieje}) and that $\aHP_I(t)$ is a polynomial (Theorem
\ref{ahppoly}).

This new definition allows generalization of results of Musta\c t\u a in the following sense:
\begin{theorem}
Let $I$ be a homogeneous radical ideal. Then, for each integer $t \geq 0$,
$$\vol(\Gamma(I) \cap \{(x_1,\dots,x_n) : x_1+\ldots+x_n \leq t\}) = \aHF_I(t).$$
Moreover, for an ideal $I$ of $r$ points in $\PP^n$ and $t \gg 0$
$$\aHF_I(t) = \aHP_I(t) = \frac{r}{n!}.$$
\end{theorem}

In the last section we present examples and briefly discuss properties of the new asymptotic
invariants.

\section{Limit of a sequence of initial ideals}
 Let $I$ be a homogeneous radical ideal in $\field[\PP^n]$, let
$\gin(I^{(m)})$ be the generic initial ideal of $I^{(m)}$, where $m$ is a nonnegative integer.
Observe that any monomial element of a monomial ideal in $\field[\PP^n]$
can be regarded as a point in $\RR^{n+1}$ via the identification $x^{\alpha} \mapsto \alpha$.
Recall also that, by Green \cite[Theorem 2.21]{Gre98} and Nagata-Zariski theorem, the following
holds.
\begin{lemma}\label{nolast}
The generators of $\gin(I^{(m)})$ for a radical homogeneous ideal $I \subset \field[\PP^n] = \field[x_1,\dots,x_{n+1}]$
does not contain the last variable $x_{n+1}$.
\end{lemma}
Thus naturally $\gin(I^{(m)})$ can be treated as a subset of $\RR^n$.

Let $L_m$ be the subsets of $\RR^n$ defined as $\{\alpha+\RR^n_{\geq 0} : x^{\alpha} \in \gin(I^{(m)})\}$.
Since $I^{(p)} \cdot I^{(q)} \subset I^{(p+q)}$, the sets $L_m$ satisfy
$$L_p+L_q \subset L_{p+q},$$
where $+$ denotes the algebraic sum of sets.
This imply, in particular, that
\begin{equation}\label{zawieranie}
kL_p\subset L_{kp}.
\end{equation}

Let $t$ be a given positive integer (unless stated otherwise, $t$ will always be an integer). In what follows we will consider the sets $L_m$
intersected with the set $\TT_{mt}:=\{(x_1,\dots,x_n) \in \RR^n_{\geq 0} : x_1+\ldots+x_n \leq mt
\}$, denoting by $L_{m,t} = L_m \cap \TT_{mt}$. Observe that if $L_{m,t}\subset \TT_m$ then
$\frac{L_{m,t}}{m}\subset \TT_t$ and

\begin{equation}\label{t-zawieranie}
kL_{p,t}\subset L_{kp,t}
\end{equation}
holds.

We want to prove the following, slightly more general theorem.
\begin{theorem}\label{limitexists}
Take a sequence $\{L_m\}_{m}$ of subsets of $\RR_{\geq 0}$ satisfying
\begin{enumerate}[i)]
\item if $\alpha \in L_m$ then $\alpha + \RR^n_{\geq 0} \subset L_m$, and
\item $L_p+L_q \subset L_{p+q}$.
\end{enumerate}
Then there exists the limit
$$\lim_{m\to \infty} \vol\left( \frac{L_{m,t}}{m}\right)=\sup_m \vol \left(\frac{L_{m,t}}{m}\right)=\vol \left( \bigcup_{m} \frac{L_{m,t}}{m} \right).$$
\end{theorem}

\begin{proof}
From \eqref{zawieranie} we have the following
sequence:
$$
L_{1,t}\subset\frac{L_{2,t}}{2}\subset\frac{L_{2\cdot 3,t}}{2\cdot
3}\subset \ldots \subset \frac{L_{m!,t}}{m!}\subset \ldots \subset
\TT_t,
$$
so the subsequence $\vol \left(\frac{L_{m!,t}}{m!}\right)$ as
monotonous and bounded, has a limit, say $g$.

We shall prove that for any nonnegative integer $m$ there exists $P$, $P>m!$, such that for
any $p\geq P$ the following holds
\begin{equation}\label{eq-trzyciagi}
\left(\frac{p-m!}{p}\right)^n\vol \left(\frac{L_{m!,t}}{m!}\right)\leq
\vol\left(\frac{L_{p,t}}{p}\right) \leq \vol\left( \frac{L_{p!,t}}{p!}\right).
\end{equation}
This claim will end the proof.

Indeed, taking \eqref{eq-trzyciagi} granted, pick any $\varepsilon >0$. Then there exists an $m$ such that
$$\vol\left(\frac{L_{p!,t}}{p!}\right)\in (g-\varepsilon, g+\varepsilon)$$
for any $p \geq m$, in particular
$$\vol \left(\frac{L_{m!,t}}{m!}\right)\in (g-\varepsilon, g+\varepsilon).$$
For such an $m$ and $p$ big enough it holds
$$\left(\frac{p-m!}{p}\right)^n\in (1-\varepsilon, 1+\varepsilon).$$
This, together with \eqref{eq-trzyciagi}, implies that for each $\varepsilon>0$,
for $p$ big enough (depending on $\varepsilon$)
$$(1-\varepsilon)(g-\varepsilon)\leq \vol \left(\frac{L_{p,t}}{p}\right)\leq  g+\varepsilon.$$

Now, to prove \eqref{eq-trzyciagi} let us write $p=r \cdot m!+d$, $0\leq d<m!$, $r,d \in \mathbb{Z}_{\geq 0}$. Observe that
$$d L_{1,t}+ \frac{p-d}{m!}L_{m!,t}\subset L_{p,t}.$$
Thus
$$\frac{d}{p}L_{1,t}+ \frac{p-d}{p}\frac{L_{m!},t}{m!}\subset \frac{L_{p,t}}{p},$$
so
$$\vol \left(\frac{p-d}{p}\frac{L_{m!,t}}{m!}\right)\leq \vol\left(\frac{d}{p}L_{1,t}+
  \frac{p-d}{p}\frac{L_{m!,t}}{m!}\right)\leq \vol \left(\frac{L_{p,t}}{p}\right).$$
Thus
\begin{gather*}
\left(\frac{p-m!}{p}\right)^n\vol \left(\frac{L_{m!,t}}{m!}\right)\leq
  \left(\frac{p-d}{p}\right)^n \vol \left(\frac{L_{m!,t}}{m!}\right) =\\
= \vol
  \left(\frac{p-d}{p}\frac{L_{m!,t}}{m!}\right)\leq \vol \left(\frac{L_{p,t}}{p}\right).
\end{gather*}
It is obvious  that $\vol \left(\frac{L_{p,t}}{p}\right)\leq \vol\left(\frac{L_{p!,t}}{p!}\right)$.
\end{proof}

\section{Volume and number of points}

Keeping the notation  $L_m=\{\alpha+\mathbb{R}^n_{\geq 0} : \alpha \in \gin(I^{(m)})\}$
define
$$\Gamma_{m} := (\RR^n_{\geq 0}\setminus L_m)$$
and
$$\Gamma_{m,t} := \Gamma_{m}\cap \mathbb{T}_{mt}.$$

Our next aim is the following lemma.
\begin{lemma}\label{volumepoints}
Denote by $\#\Gamma_{m,t}$ the number of integer-coefficient points in $\Gamma_{m,t}$. Then the following
limits exist and are equal.
$$\lim_{m\to\infty}\frac{\#\Gamma_{m,t}}{m^n}=\lim_{m\to\infty}\vol \left(\frac{\Gamma_{m,t}}{m}\right).$$
\end{lemma}

\begin{proof}
To begin with, observe that from Theorem \ref{limitexists}, the right-hand side limit exists.
It is enough to show that $\vol (\Gamma_{m,t})=\#\Gamma_{m,t}+o(m)$,
for some function $o(m)$ satysfying
$$\lim_{m\to\infty} \frac{o(m)}{m^n} = 0.$$
To show this
equality observe that a point $(k_1,\dots,k_n)\in \Gamma_{m,t}$ adds to the volume of $\Gamma_{m}$
the (unit) volume of the "cube" $[k_1,k_1+1]\times \dots \times[k_n, k_{n}+1]$. Cutting $\Gamma_m$
along $x_1+\dots+x_n=mt$ may subtract something from the volume of these ``cubes'' for which
$k_1+\dots+k_n+n>mt$. Thus, we have to estimate the number of integer-coefficient points in
$\Gamma_{m,t}$ with coordinates satisfying $k_1+\dots+k_n\geq mt-n$. This number equals at most the
number of integer-coefficient points on the following hyperplanes in $\Gamma_{m,t}$:
$$x_1+\ldots+x_n=\lfloor mt \rfloor - n, \, x_1+\ldots+x_n=\lfloor mt \rfloor -n+1, \dots, x_1+\ldots+x_n=\lfloor mt \rfloor +1.$$
The number of integer tuples $(x_1,\dots,x_n)$ satisfying $x_1+\ldots+x_n=d$ equals
$\binom{d+n-1}{n-1}$, hence $o(m)$ is bounded by
$$(n+2)\binom{\lfloor mt \rfloor +n}{n-1},$$
which completes the proof.
\end{proof}

\section{Asymptotic Hilbert function exists}

\begin{theorem}\label{ahf is a limit}
For a homogeneous radical ideal $I$ in $\field[\PP^n]$ there exists the limit
$$\lim_{m\to \infty}\frac{\HF_{I^{(m)}}(mt)}{m^n}.$$
\end{theorem}

\begin{proof}
We begin with observation (coming e.g. from Gr\"obner Bases Theory) that
$$\HF_{I^{(m)}}(mt) = \HF_{\gin(I^{(m)})}(mt).$$
Monomials of degree $mt$ which are \underline{not} in $\gin(I^{(m)})$
form a basis for $(\field[\PP^n]/\gin(I^{(m)}))_{mt}$, thus their number is
equal to $\HF_{\gin(I^{(m)})}(mt)$. By dehomogenization and Lemma \ref{nolast}
each such a monomial corresponds to exactly one monomial in $\Gamma_{m,t}$.
Hence
$$\HF_{I^{(m)}}(mt) = \#\Gamma_{m,t},$$
and dividing both sides by $m^n$ allows to use Lemma \ref{volumepoints} and pass to the limit.
\end{proof}

\begin{corollary}\label{aHFvol}
Combining the above theorem with Theorem \ref{limitexists} we obtain
$$\aHF_I(t) = \lim_{m\to\infty} \vol \left(\frac{\Gamma_{m,t}}{m}\right) = \frac{t^n}{n!} - \vol \left( \bigcup_{m} \frac{L_{m,t}}{m} \right).$$
\end{corollary}

\section{Convexity and volume}

For a set $A \subset \RR^n$, by $\conv(A)$ we denote the convex hull of $A$.
For an ideal $I$ in $\field[\PP^n]$  define
$$\Delta(I):=\overline{\bigcup  \frac{\conv(L_m)}{m}}$$
and
$$\Gamma(I):=\overline{\bigcap \left( \RR^n_{\geq 0}\setminus \frac{\conv(L_m)}{m} \right) }.$$
The convex set $\Delta(I)$ is called the \emph{limiting shape} of $I$ (compare \cite{Maypoints,DSST}).
Observe that $\Gamma(I)$ is the closure of the complement of $\Delta(I)$. Our aim is to compare asymptotic Hilbert function $\aHF_{I}$ with the volume
of restricted $\Gamma(I)$. To do this, we begin with comparing $\Delta(I)$ with the sum of the sets $L_m/m$.
In other words, we will show that taking convex hulls, as in the original definition of $\Delta(I)$,
is superfluous when passing to the limit.

\begin{lemma}\label{convexhull}
For $I$ and $L_m$ as above
$$\overline{\bigcup\frac{L_{m}}{m}}=\Delta(I).$$
\end{lemma}

\begin{proof}
It is enough to show that an element from $\conv(L_m/m)$ belongs to the left hand side of the equality.
Each such an element can be approached by a sequence of elements
of the form $\lambda a + (1-\lambda) b$, for a \underline{rational} $\lambda$ ($0\leq \lambda \leq 1$) and $a,b\in \frac{L_{m}}{m}$.
To complete the proof, write $\lambda = \frac{k}{\ell}$ where $k, \ell$ are integers. Then we may write
$$\frac{ka+(\ell-k)b}{\ell}= \frac{kma+(\ell-k)mb}{m\ell}.$$
As $ma, mb\in L_m$ and $kma\in L_{km}, (\ell-k)mb\in L_{(\ell-k)m}$, we
have
$$\frac{kma+(\ell-k)mb}{m\ell}\in \frac{L_{km}+L_{(\ell-k)m}}{m\ell}\subset \frac{L_{km+(\ell-k)m}}{m\ell}.$$
Thus $\lambda a + (1-\lambda) b \in \frac{L_{m\ell}}{m\ell}$, which completes the proof.
\end{proof}

\begin{lemma}\label{boundary}
With the notation as in Lemma \ref{convexhull} we have
$$\vol \left(\bigcup\frac{L_{m,t}}{m}\right) = \vol \left(\overline{\bigcup\frac{L_{m,t}}{m}}\right).$$
\end{lemma}

\begin{proof}
It is enough to
show that the volume of the boundary of the set
$A=\bigcup\frac{L_m}{m}$ is zero. Change (rotate) the coordinates
in $\mathbb{R}^n$ in such a way that the  line $x_1=x_2=\ldots=x_n$
becomes the new $x_n$ axis. Then the boundary of $A$ is the limit
of the graphs of Lipschitz functions (all with the Lipschitz constant greater or equal to 1)
describing the boundary of
$A_m$. Thus, the boundary of $A$ is a graph of a function, and
hence has volume zero, which proves the lemma.
\end{proof}

The following theorem establishes connection between $\aHF$ and $\Gamma$. Recall that
$\TT_t = \{ (x_1,\dots,x_n) : x_1 + \ldots + x_n \leq t \}$.
\begin{theorem}\label{ahfvol}
For a homogeneous radical ideal $I \subset \field[\PP^n]$ we have
$$\aHF_{I}(t) = \vol (\Gamma(I) \cap \TT_t).$$
\end{theorem}

\begin{proof}
By Corollary \ref{aHFvol}
$$\aHF_{I}(t) = \frac{t^n}{n!} - \vol \left( \bigcup \frac{L_{m,t}}{m} \right).$$
By Lemma \ref{boundary} this is equal to the volume of closure of the sum of sets presented above,
which by Lemma \ref{convexhull} is equal to
$$\frac{t^n}{n!} - \vol (\Delta(I) \cap \TT_t).$$
The set $\Delta(I) \cap \TT_t$ is convex, so (compare proof of Lemma \ref{boundary}) the volume of its restricted complement
is equal to the volume of the closure of the restricted complement. Therefore
$$\vol(\Gamma(I) \cap \TT_t) = \frac{t^n}{n!} - \vol(\Delta(I) \cap \TT_t)$$
and the claim follows.
\end{proof}

\section{Asymptotic Hilbert polynomial}

To show that asymptotic Hilbert polynomial exists and has expected properties, we need a property
called \emph{linearly bounded symbolic regularity}, LBSR for short. By $\reg(I)$ we mean the Castelnuovo-Mumford regularity
of $I$.
\begin{definition}\label{LBSR}
Let $I$ be a homogeneous ideal in $\field[\PP^n]$. We say that $I$ satisfies linearly bounded symbolic regularity,
or \emph{is LBSR} for short, if there exists constants $a,b>0$, such that
$$\reg(I^{(m)}) \leq am+b \text{ for $m$ big enough.}$$
\end{definition}

So far, we do not know any example of a homogeneous ideal, which is \underline{not} LBSR. However,
a lot of effort has been given to prove that every ideal is LBSR, but with limited success.
We bring here some of results, listing the ideals, which are certainly LBSR (compare \cite{Herzog}):
\begin{itemize}
\item $\dim(R/I) \leq 2$ (which means $\dim(V(I))\leq 1$ as a set in $\PP^n$);
\item ideals with singular locus of dimension zero;
\item in particular: ideals defining disjoint linear subspaces and ideals defining lines intersecting in points, which we use in examples;
\item monomial ideals.
\end{itemize}
We recall also that ordinary powers $I^m$ are linearly bounded, as showed by Swanson \cite{Swan}.

\begin{theorem}\label{ahp istnieje}
If $I$ is a radical homogeneous ideal with linearly bounded symbolic regularity then $\aHP_{I}$ exists, and for $t$ big enough $\aHP_{I}(t)=\aHF_{I}(t)$.
\end{theorem}

\begin{proof}
Recall that for $t > \reg(I)$ we have $\HF_{I}(t) = \HP_{I}(t)$. For $t > a+b$ (where $a$, $b$ as
in Definition \ref{LBSR}) we have $mt > am+b$ and
$\HP_{I^{(m)}}(mt) = \HF_{I^{(m)}}(mt)$. We pass to the limit using Theorem \ref{ahf is a limit}.
\end{proof}

\begin{corollary}\label{ahpvol}
Observe that, by Theorem \ref{ahfvol}, for an LBSR ideal $I$ we have
$$\aHP_{I}(t) = \vol (\Gamma \cap \TT_t), \quad t \gg 0.$$
\end{corollary}

Now we will show that $\aHP_I$ indeed can be called a polynomial.
\begin{theorem}\label{ahppoly}
For an ideal $I$ as above, the asymptotic Hilbert polynomial $\aHP_I$ is a polynomial.
\end{theorem}

\begin{proof}
Observe that $\aHP_I$ is a pointwise limit of polynomials $f_m(t)$,
$$f_m(t) := \frac{\HP_{I^{(mt)}}(mt)}{m^n}$$
with common bound $\deg_t(f_m(t))\leq n$ for degree with respect to $t$ (this bound is given by the dimension of the ambient space;
in fact, we can bound this degree by a dimension of $V(I)$). Let $a_{m,j}$ be the coefficient
standing by $t^j$ in $f_m$. Choose $n+1$ distinct points $t_0,\dots,t_n$ big enough, as in the previous proof.
The coefficients $a_{m,j}$ depends polynomially (by Lagrange interpolation), hence continuously,
on values $f_m(t_0),\dots,f_m(t_n)$. Therefore each of the sequences $\{a_{m,0}\}_{m}$, $\{a_{m,1}\}_{m}, \dots, \{a_{m,n}\}_{m}$
is convergent, say to $a_0,\dots,a_n$, respectively. Then $f=a_0+a_1t+\ldots+a_nt^n$ is a pointwise limit of $\{f_m\}$, hence
equal to $\aHP_I$.
\end{proof}

\begin{remark}
From obvious reasons, we have defined $\aHF_I(t)$ for an integer $t$, thus the equality $\aHP_I(t)
= \vol(\Gamma(I) \cap \TT_t)$
 has been proved for integer values of $t$. One may ask if it also
holds for every positive $t$ (big enough), and the answer is positive. We do not present the full
proof here. It suffices to consider rational $t$, and for such a $t$ there exists a subsequence of
$\{\HF_{I^{(m)}}(mt)\}_m$ with integers values of $mt$, which gives the result.
\end{remark}

\section{Asymptotic Hilbert polynomial for subspaces}

Let $L$ be a sum of $s$ disjoint linear subspaces of dimension $r$ in $\mathbb{P}^n$ (called a \emph{flat};
by a \emph{fat flat} we denote such subspaces with multiplicities), as in \cite{DHST}. Let $I$ be the
ideal of $L$. Following \cite{DHST} define
$$P_{n,r,s,m}(t):=\binom{t+n}{n}-\HP_{I^{(m)}}(t).$$
Substitute $t$ by $mt$ into $P_{n,r,m,s}(t)$ and regard it as a polynomial in $m$ (this is indeed a polynomial, see \cite{DHST}).
The leading term of this polynomial is in \cite{DHST} denoted by $\Lambda_{n,r,s}(t).$

We have the following:

\begin{theorem}\label{fatflatsy}
For $I$ as above
$$\aHP_I(t)=\frac{t^n}{n!}-\Lambda(n,r,s)(t).$$
\end{theorem}

\begin{proof}
From \cite[Lemma 2.1]{DHST}
$$\HP_{I^{(m)}}(t) = \sum_{0\leq i<m} \binom{t-i+r}{r}\binom{i+n-r-1}{n-r-1},$$
hence $\HP_{I^{(m)}}(mt)$ is a polynomial also with respect to $m$.

Now observe that for any polynomial $f(m)=a_0+a_1m+\ldots+a_nm^n+\ldots+a_dm^d$,
$$\lim_{m \to \infty} \frac{f(m)}{m^n} = \begin{cases}
a_n, & d \leq n \\
\pm\infty, & d > n.
\end{cases}$$

Take $f(m) := \HP_{I^{(m)}}(mt)$, which is a polynomial in $m$. Since $\aHP_I(t)$ exists and is non-zero,
the degree of $f(m)$ is equal to $n$. Therefore its leading term can be obtained with limiting $f(m)/m^n$,
which completes the proof.
\end{proof}

\begin{corollary}
For an ideal $I$ of $r$ distinct points in $\PP^n$, we have
$$\aHP_I(t) = \frac{r}{n!} = \vol(\Gamma(I)).$$
Thus Theorem \ref{ahp istnieje} gives a generalization to
results obtained by Musta\c t\u a.
\end{corollary}

\begin{proof}
Since the number of conditions imposed by $r$ points in $\PP^n$ on forms of
sufficiently big degree $t$ is constant and equal $r\binom{m+n-1}{n} = \HP_{I^{(m)}}(t)$,
we pass to the limit to obtain $\aHP_I=r/n!$. The second equality follows from Corollary \ref{ahpvol}.
\end{proof}

\section{Examples}

\begin{example}
Let $I$ be the ideal of two generic lines in $\PP^3$. Then the limiting shape
$\Gamma(I) \subset \RR^3$ is a sum of a cylinder over a triangle with vertices
at $(0,0,0)$, $(1,0,0)$, $(0,2,0)$, and a pyramid over the triangle
$(1,0,0)$, $(2,0,0)$, $(0,2,0)$ with vertex at $(1,0,1)$. The asymptotic Hilbert polynomial
is equal $\aHP_{I}(t) = t - 2/3$.
\end{example}
\begin{figure}[H]
\begin{tikzpicture}[line cap=round,line join=round,>=triangle 45,x=1.0cm,y=1.0cm]
\clip(-2.8,-0.7) rectangle (7.0600000000000005,6.3);
\fill[color=zzttqq,fill=zzttqq,fill opacity=0.1] (1.64,0.84) -- (0.24,1.76) -- (3.88,2.16) -- cycle;
\fill[color=zzttqq,fill=zzttqq,fill opacity=0.1] (0.26,4.84) -- (3.9,5.12) -- (3.88,2.16) -- (0.24,1.76) -- cycle;
\fill[color=zzttqq,fill=zzttqq,fill opacity=0.1] (0.26,4.84) -- (1.66,3.6) -- (1.64,0.84) -- (0.24,1.76) -- cycle;
\fill[color=zzttqq,fill=zzttqq,fill opacity=0.1] (1.66,3.6) -- (3.9,5.12) -- (3.88,2.16) -- (1.64,0.84) -- cycle;
\fill[color=zzttqq,fill=zzttqq,fill opacity=0.1] (1.651739564190076,2.460059858230507) -- (1.64,0.84) -- (2.9474116305587232,-0.019156214367161084) -- cycle;
\fill[color=zzttqq,fill=zzttqq,fill opacity=0.1] (1.651739564190076,2.460059858230507) -- (3.88,2.16) -- (2.9474116305587232,-0.019156214367161084) -- cycle;
\fill[color=zzttqq,fill=zzttqq,fill opacity=0.1] (1.64,0.84) -- (2.9474116305587232,-0.019156214367161084) -- (3.88,2.16) -- cycle;
\draw [color=zzttqq] (1.64,0.84)-- (0.24,1.76);
\draw [color=zzttqq] (0.24,1.76)-- (3.88,2.16);
\draw [color=zzttqq] (3.88,2.16)-- (1.64,0.84);
\draw (1.64,0.84)-- (1.66,3.6);
\draw (0.24,1.76)-- (0.26,4.84);
\draw (3.88,2.16)-- (3.9,5.12);
\draw [color=zzttqq] (0.26,4.84)-- (3.9,5.12);
\draw [color=zzttqq] (3.9,5.12)-- (3.88,2.16);
\draw [color=zzttqq] (3.88,2.16)-- (0.24,1.76);
\draw [color=zzttqq] (0.24,1.76)-- (0.26,4.84);
\draw [color=zzttqq] (0.26,4.84)-- (1.66,3.6);
\draw [color=zzttqq] (1.66,3.6)-- (1.64,0.84);
\draw [color=zzttqq] (1.64,0.84)-- (0.24,1.76);
\draw [color=zzttqq] (0.24,1.76)-- (0.26,4.84);
\draw [color=zzttqq] (1.66,3.6)-- (3.9,5.12);
\draw [color=zzttqq] (3.9,5.12)-- (3.88,2.16);
\draw [color=zzttqq] (3.88,2.16)-- (1.64,0.84);
\draw [color=zzttqq] (1.64,0.84)-- (1.66,3.6);
\draw [color=zzttqq] (1.651739564190076,2.460059858230507)-- (1.64,0.84);
\draw [color=zzttqq] (1.64,0.84)-- (2.9474116305587232,-0.019156214367161084);
\draw [color=zzttqq] (2.9474116305587232,-0.019156214367161084)-- (1.651739564190076,2.460059858230507);
\draw [color=zzttqq] (1.651739564190076,2.460059858230507)-- (3.88,2.16);
\draw [color=zzttqq] (3.88,2.16)-- (2.9474116305587232,-0.019156214367161084);
\draw [color=zzttqq] (2.9474116305587232,-0.019156214367161084)-- (1.651739564190076,2.460059858230507);
\draw [color=zzttqq] (1.64,0.84)-- (2.9474116305587232,-0.019156214367161084);
\draw [color=zzttqq] (2.9474116305587232,-0.019156214367161084)-- (3.88,2.16);
\draw [color=zzttqq] (3.88,2.16)-- (1.64,0.84);
\draw [dash pattern=on 5pt off 5pt] (0.25999999999999995,4.839999999999999)-- (0.26,6.18);
\draw [dash pattern=on 5pt off 5pt] (3.9000000000000004,5.12)-- (3.88,6.24);
\draw [dash pattern=on 5pt off 5pt] (1.66,3.6000000000000005)-- (1.66,6.18);
\begin{scriptsize}
\draw [fill=black] (0.24,1.76) circle (1.5pt);
\draw [fill=black] (1.64,0.84) circle (1.5pt);
\draw [fill=black] (2.9474116305587232,-0.019156214367161084) circle (1.5pt);
\draw [fill=black] (3.88,2.16) circle (1.5pt);
\draw [fill=black] (1.651739564190076,2.460059858230507) circle (1.5pt);
\end{scriptsize}
\end{tikzpicture}
\caption{$\Gamma(I)$ for two general lines}\label{fig: Gamma}
\end{figure}
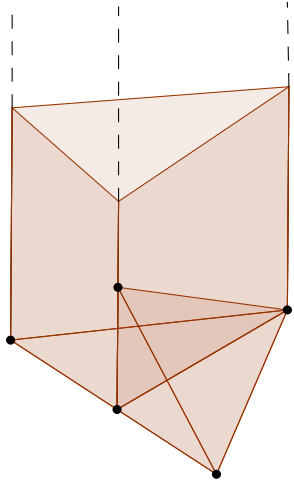

\begin{proof}
First, to compute the asymptotic Hilbert polynomial of $I$ it is enough to observe that
$V(I)$ is a fat flat, hence $\aHP_I$ can be computed using Theorem \ref{fatflatsy} and \cite{DHST}.
Observe also that, for generic lines, computing $\gin(I)$ equals, in fact, computing $\iin(I)$.
Next, we prove that
\begin{equation}
\label{gin2lines}
\iin(I) \supset (x_1x_3,x_2^2,x_1x_2,x_1^2).
\end{equation}
To do this observe that each of the monomials in \eqref{gin2lines}
is bigger (with respect to degrevlex order) than any of the following monomials:
\begin{equation}\label{smallermonomials}
x_2x_3, x_3^2, x_1x_4, x_2x_4, x_3x_4, x_4^2.
\end{equation}
Take one of the monomials from \eqref{gin2lines} and all monomials from \eqref{smallermonomials}
(there are seven of them in total, all of degree 2), and let $W$ be the vector space of polynomials
based on these monomials. Vanishing along a line imposes at most $3$ condition on forms of degree
$2$, so there must be an element $f \in W \cap I$. Assume that a coefficient standing by a monomial
from \eqref{gin2lines} is zero. Such an $f$, restricted to the hyperplane $H = \{x_4=0\}$, is
either zero, or of the form $g := ax_2x_3+bx_3^2$. The
first possibility means that at least one of lines lies in $H$, a contradiction with generality.
The curve $\{g=0\}$ on $H$ passes through two general points on $H$, the traces of general lines on
$H$. Observe that $g$ describes the sum of two lines --- one, given by the equation $x_3=0$ is fixed, hence cannot pass through a
general point, the second, given by $ax_2+bx_3=0$ belongs to the pencil of lines through $(1:0:0)$, thus cannot pass through two general points at the same time.

Having proved \eqref{gin2lines}, it follows that (by definition) $\Delta(I)$ is contained in the convex hull
of the set $\{(1,0,1), (0,2,0), (1,1,0), (2,0,0)\}$. The complement of this set is exactly the shape we claim to be
the $\Gamma(I)$, and we know that it is contained in $\Gamma(I)$. Computing the volume
of our shape restricted to $\TT_t$ for $t$ big enough, we obtain that it is equal to $\aHP_I(t)$, which
completes the proof.
\end{proof}

\begin{example}
We will show that asymptotic version of Hilbert polynomial is a subtle invariant, and cannot be driven
directly out of the Hilbert polynomial. To do this, consider two ideals. Let $I$ be the ideal
of two lines (general enough) intersecting at a general point, and an additional general point. Let $J$ be the ideal
of two general lines (both ideals are considered in $\field[\PP^3]$).
The number of conditions imposed by vanishing along a line on forms of degree $t$ is $t+1$, hence
$$\HP_J(t) = 2\HP_{\text{ideal of a line}}(t) = 2t+2.$$
In the case of intersecting lines, vanishing along the second needs one condition less (it is already imposed by the
first line). Therefore
$$\HP_I(t) = (t+1) + (t) + 1 = 2t+2,$$
the last one being the condition imposed by the additional point. Thus $\HP_I(t)=\HP_J(t)$.

Passing to $\aHP$, recall that $\aHP_J(t)=t-2/3$, but $\aHP_I(t)=t-5/6$. To prove this, we first
compute $\aHP$ for ideal of two intersecting lines, and then add $\aHP$ of a single point in $\PP^3$, which
is $1/6$. Computing $\aHP$ for an ideal $L$ of two intersecting lines is the most technical part here,
which we omit. The idea is to compute
$\HP_{L^{(m)}}(t)$ for each $m$ and $t$, that is to find, how many conditions are imposed by vanishing
along these lines on forms of degree $t$. To do this, we change the coordinate system such that lines
are described by $x_1=x_2=0$ and $x_1=x_3=0$, and compute monomials that ``survive'' vanishing.
The result yields
$$\HP_{L^{(m)}}(t) = (m^2+m)t - m^3+\frac{1}{2}m^2+\frac{3}{2}m.$$
Replacing $t$ by $mt$, dividing by $m^3$ and passing to the
limit with $m$ gives the result. It is worth to mention here, that $\Gamma(L)$ is just the cylinder over a
triangle with vertices $(0,0,0)$, $(1,0,0)$, $(0,2,0)$, and $\Gamma(I)$ is equal to $\Gamma(L)$ plus a
pyramid over the triangle $(1,0,0)$, $(3/2,0,0)$, $(0,2,0)$ with vertex at $(1,0,1)$.
\end{example}

\paragraph*{\emph{Acknowledgements.}}
   We would like to thank  E. Tutaj and T. Szemberg for helpful discussions.

\bigskip \small

\bigskip
   Marcin Dumnicki, Halszka Tutaj-Gasi\'nska,
   Jagiellonian University, Institute of Mathematics, {\L}ojasiewicza 6, PL-30-348 Krak\'ow, Poland

\nopagebreak
   \textit{E-mail address:} \texttt{Marcin.Dumnicki@im.uj.edu.pl}

   \textit{E-mail address:} \texttt{Halszka.Tutaj@im.uj.edu.pl}

\bigskip
   Justyna Szpond,
   Instytut Matematyki UP,
   Podchor\c a\.zych 2,
   PL-30-084 Krak\'ow, Poland

\nopagebreak
   \textit{E-mail address:} \texttt{szpond@up.krakow.pl}

\end{document}